\documentclass[12pt]{amsart}
\usepackage{amssymb,latexsym,eufrak,amsmath,amscd, graphicx}
\usepackage[colorlinks=true,pagebackref,hyperindex]{hyperref}

\usepackage{amsfonts}
\usepackage{amscd}

\usepackage{xypic}

\renewcommand{\phi}{\varphi}
\renewcommand{\epsilon}{\varepsilon}
\renewcommand{\theta}{\vartheta}

\def\ZZ{{\mathbf Z}}

\def\CC{{\mathbf C}}

\def\RR{{\mathbf R}}
\def\QQ{{\mathbf Q}}

\def\cJ{\mathcal{J}}

\def\cF{\mathcal{F}}

\def\cG{\mathcal{G}}

\def\cO{\mathcal{O}}

\def\fra{\mathfrak{a}}
\def\frb{\mathfrak{b}}

\def\frm{\mathfrak{m}}
\def\frn{\mathfrak{n}}

\def\frp{\mathfrak{p}}

 \DeclareMathOperator{\Spec}{Spec}
 
 \DeclareMathOperator{\lct}{lct}
 
 \DeclareMathOperator{\fpt}{fpt}

\newcommand{\llbracket}{[\negthinspace[}
\newcommand{\rrbracket}{]\negthinspace]}

\newtheorem{lemma}{Lemma}[section]
\newtheorem{theorem}[lemma]{Theorem}
\newtheorem{corollary}[lemma]{Corollary}
\newtheorem{proposition}[lemma]{Proposition}

\theoremstyle{definition}
\newtheorem{definition}[lemma]{Definition}
\newtheorem{remark}[lemma]{Remark}

\newtheorem{problem}[lemma]{Problem}

\newtheorem{example}[lemma]{Example}

\theoremstyle{remark}
\newtheorem*{remark*}{Remark}
\newtheorem*{note*}{Note}

\begin{document}

\title{Estimates for $F$-jumping numbers and bounds for Hartshorne-Speiser-Lyubeznik
numbers}

\thanks{2010\,\emph{Mathematics Subject Classification}.
 Primary 13A35; Secondary 14B15, 14F18.
\newline The first author was partially supported by
 NSF grant DMS-1068190 and
  a Packard Fellowship.
  \newline The second author was partially supported by NSF grant DMS-1068946.}

\author[M.~Musta\c{t}\u{a}]{Mircea~Musta\c{t}\u{a}}
\address{Department of Mathematics, University of Michigan,
Ann Arbor, MI 48109, USA}
\email{mmustata@umich.edu}

\author[W.~Zhang]{Wenliang~Zhang}
\address{Department of Mathematics, University of Michigan,
Ann Arbor, MI 48109, USA}

\curraddr{Department of Mathematics, University of Nebraska, Lincoln, NE 68588-0130,
USA}
\email{wzhang15@unl.edu}

\begin{abstract}
Given an ideal $\fra$ on a smooth variety in characteristic zero, we estimate
the $F$-jumping numbers of the reductions of $\fra$ to positive characteristic in terms
of the jumping numbers of $\fra$ and the characteristic. We apply one of our estimates
to bound the Hartshorne-Speiser-Lyubeznik invariant for the reduction to positive characteristic of a hypersurface singularity.
\end{abstract}

\maketitle

\markboth{M.~MUSTA\c{T}\u{A} AND W.~ZHANG}{ESTIMATES FOR $F$-JUMPING 
NUMBERS AND BOUNDS FOR HSL NUMBERS}

\section{Introduction}
Let $\fra$ be a nonzero ideal on a 
smooth, irreducible variety $X$ over an algebraically closed field $k$. 
A fundamental invariant of the singularities of the subscheme defined by $\fra$ is the
\emph{log canonical threshold} $\lct(\fra)$. This can be defined in terms of either divisorial valuations or, when working over the complex numbers,
integrability conditions.
On the other hand, by considering models for $X$ and $\fra$
over a $\ZZ$-algebra of finite type $A\subset k$, one can take reductions $X_s$ and 
$\fra_s$
to positive characteristic, for all closed points $s\in\Spec A$. 
Using the Frobenius morphism, Takagi and Watanabe defined in 
\cite{TW} an analogue of the log canonical threshold
in this setting, the $F$-\emph{pure threshold} $\fpt(\fra_s)$. A problem that has attracted a lot of interest is the relation between $\lct(\fra)$ and $\fpt(\fra_s)$. 

It follows from the work of Hara and Yoshida \cite{HY} that after possibly replacing $A$ by a localization $A_a$, we may assume that $\lct(\fra)\geq\fpt(\fra_s)$ for all closed points 
$s\in\Spec A$. Moreover, for every $\epsilon>0$, there is an open subset $U_{\epsilon}
\subseteq\Spec A$
such that $\fpt(\fra_s)>\lct(\fra)-\epsilon$ for all $s\in U_{\epsilon}$. One can see that even in very simple examples, one can not take $U_{\epsilon}$ to be independent of $\epsilon$.
On the other hand, it is expected that there is a Zariski dense set of closed points 
$s\in\Spec A$ such that $\lct(\fra)=\fpt(\fra_s)$ (see \cite{MS}). 
As a consequence of our main results, we give an effective estimate for the difference between
the log canonical threshold of $\fra$ and the $F$-pure threshold of $\fra_s$
(see Corollaries~\ref{cor_lct} and \ref{cor2_thm_main2} below).

\noindent{\bf Theorem A}. With the above notation, after possibly replacing $A$ by a localization 
$A_a$, the following hold:
\begin{enumerate}
\item[(i)] There is $C>0$ such that $\lct(\fra)-\fpt(\fra_s)\leq
\frac{C}{{\rm char}(k(s))}$ for every closed point $s\in \Spec A$.
\item[(ii)] Assuming that $\fra$ is locally principal, there is a positive integer $N$ such that
$$\lct(\fra)-\fpt(\fra_s)\geq \frac{1}{{\rm char}(k(s))^N}$$ for every closed point
$s\in\Spec A$ for which $\fpt(\fra_s)\neq\lct(\fra)$.
\end{enumerate}

In fact, we prove similar estimates for the higher jumping numbers, that we now describe.
Recall that
if $\fra$ and $X$ are as above, then one associates to $\fra$ and to every
$\lambda\in\RR_{\geq 0}$ the multiplier ideal $\cJ(\fra^{\lambda})$ of $\cO_X$. 
These ideals have found a lot of applications in the study of higher-dimensional varieties,
due to the fact that they measure the singularities of the subscheme defined by $\fra$ in a way 
that is relevant to vanishing theorems (see \cite[\S 9]{positivity}). 
The multiplier ideals can be defined using either divisorial valuations, or  integrability conditions (when we work over $\CC$). All multiplier ideals can be computed from a log resolution
of $\fra$,
and this description immediately implies that there is an unbounded sequence of positive rational numbers $\lambda_1<\lambda_2<\ldots$ such that 
$$\cJ(\fra^{\lambda})=\cJ(\fra^{\lambda_i})\supsetneq
\cJ(\fra^{\lambda_{i+1}})\,\,\text{for all}\,\,i\geq 0\,\,\text{and all}\,\,
\lambda\in [\lambda_i,\lambda_{i+1})$$
(with the convention $\lambda_0=0$). The rational numbers $\lambda_i$, with  $i\geq 1$,
are the \emph{jumping numbers} of $\fra$. The smallest such number $\lambda_1$ can be described as
the smallest $\lambda$ such that $\cJ(\fra^{\lambda})\neq\cO_X$; this is the
the log canonical threshold $\lct(\fra)$.

Suppose now that we choose models of $X$ and $\fra$ over $A$ as before, and we consider
the reduction $\fra_s$, where $s\in\Spec A$ is a closed point. 
Hara and Yoshida introduced in \cite{HY}
the (generalized) \emph{test ideals} $\tau(\fra_s^{\lambda})$. While giving an analogue 
of multiplier ideals in the positive characteristic setting, they are defined by very different methods (the original definition in \cite{HY} involves a generalization of the theory of tight closure, due to Hochster and Huneke). One can show that in this case, too, there
is an unbounded, strictly increasing sequence of positive rational numbers 
$\alpha_i=\alpha_i(s)$ for $i\geq 0$, with $\alpha_0=0$, such that 
$$\tau(\fra^{\lambda})=\tau(\fra^{\alpha_i})\supsetneq\tau(\fra^{\alpha_{i+1}})\,\,
\text{for all}\,\,i\geq 0\,\,\text{and all}\,\,\lambda\in [\alpha_i,\alpha_{i+1}).$$
The rational numbers $\alpha_i$, with $i\geq 1$, are the $F$-\emph{jumping
numbers} of $\fra$. The smallest $F$-jumping number $\alpha_1$ can be described as the smallest $\lambda$
such that $\tau(\fra_s^{\lambda})\neq\cO_{X_s}$; this is the $F$-pure threshold
$\fpt(\fra_s)$.
We mention that unlike in the case of multiplier ideals, both the
rationality of the $\alpha_i$, and the fact that they are unbounded, is nontrivial
(see \cite{BMS2}). 

The comparison between $\lct(\fra)$ and $\fpt(\fra_s)$ comes from a relation
between the multiplier ideals of $\fra$ and the test ideals of $\fra_s$, proved in 
\cite{HY}. This says that after possibly replacing $A$ by a localization $A_a$, we may assume that
$$\tau(\fra_s^{\lambda})\subseteq\cJ(\fra^{\lambda})_s$$
for all closed points $s\in\Spec A$. Furthermore, given any $\lambda\in\RR_{\geq 0}$,
there is an open subset $V_{\lambda}\subseteq\Spec A$ such that $\tau(\fra^{\lambda}_s)=
\cJ(\fra^{\lambda})_s$ for every closed point $s\in V_{\lambda}$. This set, in general, depends on $\lambda$. On the other hand,
it is expected that there is a Zariski dense set of closed points $s\in \Spec A$ such that
$\tau(\fra^{\lambda}_s)=
\cJ(\fra^{\lambda})_s$ for every $\lambda$ (see \cite{MS}). 
We can now state our main results concerning jumping numbers (see
 Theorem~\ref{thm_main1}, 
Theorem~\ref{thm_main2}, and Corollary~\ref{cor_thm_main2} below).

\noindent{\bf Theorem B}. 
With the above notation, given $\lambda\in\QQ_{>0}$, after possibly replacing $A$ by a localization $A_a$,  the following hold:
\begin{enumerate}
\item[(i)] There is $C>0$ such that for every closed point $s\in\Spec A$ with 
${\rm char}(k(s))=p_s$, we have
$$\cJ(\fra^{\lambda-\frac{C}{p_s}})_s=\tau(\fra_s^{\lambda-\frac{C}{p_s}}).$$
In particular, if $\lambda$ is a jumping number of $\fra$ and $\lambda'$ is the largest jumping number smaller than $\lambda$ (with the convention $\lambda'=0$ if 
$\lambda=\lct(\fra)$), then we may assume that for every 
$s$ as above, there is an $F$-jumping number $\mu\in (\lambda',\lambda]$ for
$\fra_s$, and for every such $\mu$, we have
$\lambda-\mu\leq\frac{C}{p_s}$.
\item[(ii)] Assuming that $\fra$ is locally principal, there is a positive integer $N$ such that 
for every $F$-jumping number $\mu<\lambda$ of $\fra_s$, we have
$\lambda-\mu\geq\frac{1}{p_s^N}$.
\end{enumerate}

We deduce the assertion in (ii) from the description of test ideals in \cite{BMS2} and an observation from \cite{BMS1}. The more involved assertion in (i) follows using the methods
introduced by Hara and Yoshida in \cite{HY}. The statements in Theorem~A then follow
by applying Theorem~B to $\lambda=\lct(\fra)$. 

We apply assertion (ii) in Theorem B to
Hartshorne-Speiser-Lyubeznik numbers, as follows. Recall that given a Noetherian local ring 
$(S,\frn)$ of characteristic $p>0$, a $p$-linear structure on an $S$-module $M$
is an additive map $\phi\colon M\to M$ such that $\phi(az)=a^p\phi(z)$ for all $a\in S$ and 
$z\in M$. If $M$ is Artinian, then by a theorem due to Hartshorne and
Speiser \cite{HS} and Lyubeznik \cite{Lyubeznik}, the non-decreasing sequence of 
$S$-submodules
$$N_i:=\{z\in M\mid\phi^i(z)=0\}\subseteq M$$
 is eventually stationary. The Hartshorne-Speiser-Lyubeznik
number of $(M,\phi)$ is the smallest $\ell$ such that $N_{\ell}=N_{\ell+j}$ for all $j\geq 1$. 

We are interested in the case when $S=R/(f)$, for a regular local ring $R$
of positive characteristic and a nonzero noninvertible $f\in R$. Let $d=\dim(S)$. In this case the injective hull
$E_S$ of $S/\frn$ over $S$ can be identified with the top cohomology
module $H^{d}_{\frn}(S)$, and therefore carries a canonical $p$-linear structure $\Theta$
induced by functoriality from the Frobenius action on $S$. If we are in a setting where
the test ideals of $R$ are defined (for example, when $R$ is essentially of finite type over
a perfect field), then the Hartshorne-Speiser-Lyubeznik number of
$(E_S,\Theta)$ is equal to the smallest positive integer $\ell$ such that
$$\tau(f^{1-\frac{1}{p^{\ell}}})=\tau(f^{1-\frac{1}{p^{\ell+j}}})$$
for every $j\geq 1$. 
If we are in the setting of Theorem~B (ii), we obtain the following
(see Theorem~\ref{thm_main3} below).

\noindent{\bf Theorem C}. If $X$, $\fra$ and $A$ are as in Theorem~B, with $\fra$
locally principal, and if $Z$ is the subscheme defined by $\fra$, then there is a positive 
integer $N$ such that for every closed point $s\in\Spec A$ and every point in the fiber
$Z_s$ of $Z$ over $s$, the Hartshorne-Speiser-Lyubeznik number of
$(E_{\cO_{Z_s,x}},\Theta)$ is bounded above by $N$.

We also give an example to illustrate that in the above theorem, 
even after possibly replacing $A$ by a localization, we can not take $N=1$
(this gives a negative answer to a question of M.~Katzman).

The paper is structured as follows. In \S 2 we recall the definitions of multiplier ideals
and test ideals, as well as the framework for reducing from characteristic zero to positive characteristic. In \S 3 we explain how to get upper bounds for the jumping numbers
of an ideal in positive characteristic. In particular, we prove part (ii) in Theorems~A and B.
In \S 4 we describe how to use the methods from \cite{HY} to get lower bounds for the
$F$-jumping numbers of the reductions to positive characteristic of an ideal defined in characteristic zero. This gives part (i) in Theorems~A and B. In \S 5 we discuss the
Hartshorne-Speiser-Lyubeznik numbers and their connection with $F$-jumping numbers.
The last section contains some examples.

\subsection*{Acknowledgment} 
Our project started during the AIM workshop on ``Relating test ideals and mutiplier ideals".
We would like to thank Karl Schwede and Kevin Tucker for organizing this event, and
AIM for providing a stimulating research environment. The second author would also like to thank Mordechai Katzman for helpful discussions during the AIM workshop.

\section{Review of multiplier ideals and test ideals}

In this section we recall the definitions of multiplier ideals and test ideals, and review the results
connecting these ideals via reduction mod $p$. For simplicity, we only consider the case of
smooth ambient algebraic varieties\footnote{While the results in \S 4 work more generally, for the
upper bounds in \S 3 we will need to restrict anyway to smooth ambient varieties.}.
We start by discussing the multiplier ideals
in characteristic zero.

\subsection{Multiplier ideals}
Let $X$ be a smooth scheme of finite type over an algebraically closed field $k$ of characteristic zero.
Suppose that $\fra$ is an ideal\footnote{All ideal sheaves are assumed to be coherent.}
of $\cO_X$ which is everywhere nonzero (in other words, its restriction to every connected component of $X$ is nonzero). A
\emph{log resolution} of $\fra$ is a projective birational morphism
$\pi\colon Y\to X$, with $Y$ smooth, such that $\fra\cdot\cO_Y$ is the ideal defining
a divisor $D$ on $Y$, and such that $D+K_{Y/X}$ is a simple normal crossing divisor.
Here $K_{Y/X}$ is the relative canonical divisor, an effective divisor supported on the exceptional 
locus of $\pi$, such that $\cO_Y(K_{Y/X})\simeq\omega_Y\otimes \pi^*(\omega_X)^{-1}$
(recall that for a smooth scheme $W$ over $k$, one denotes by $\omega_W$ the line bundle of
top differential forms on $W$).

Given such a log resolution, we define for every $\lambda\in\RR_{\geq 0}$
the \emph{multiplier ideal} of $\fra$ of exponent $\lambda$ by
$$\cJ(\fra^{\lambda}):=\pi_*\cO_Y(K_{Y/X}-\lfloor\lambda D\rfloor).$$
Here, for a divisor with real coefficients $F=\sum_i\alpha_iF_i$, we put
$\lfloor F\rfloor:=\sum_i\lfloor \alpha_i\rfloor F_i$, where $\lfloor u\rfloor$ denotes the largest integer $\leq u$.
Note that since $K_{Y/X}$ is effective and supported on the exceptional locus, we have
$\pi_*\cO_Y(K_{Y/X})=\cO_X$, hence $\cJ(\fra^{\lambda})$ is indeed an ideal of $\cO_X$. 

We now review some basic properties of multiplier ideals. For proofs of these facts, and for 
a detailed introduction to this topic, see \cite[\S 9]{positivity}. One can show that the definition 
is independent of the choice of log resolution. A few things are straightforward from definition:
if $\lambda<\mu$, then $\cJ(\fra^{\mu})\subseteq\cJ(\fra^{\lambda})$. Furthermore, 
for every $\lambda\in\RR_{\geq 0}$, there is $\epsilon>0$ such that 
$\cJ(\fra^{\lambda})=\cJ(\fra^{\mu})$ for every $\mu$ with $\lambda\leq\mu\leq\lambda+\epsilon$.
One says that $\lambda>0$ is a \emph{jumping number} of $\fra$ if
$\cJ(\fra^{\lambda})$ is strictly contained in $\cJ(\fra^{\mu})$ for every $\mu<\lambda$. 

It follows from definition that if $D=\sum_{i=1}^ra_iD_i$, then for every jumping number
$\lambda$ of $\fra$, there is $i$ such that $\lambda a_i$ is an integer. In particular,
we see that the set of jumping numbers of $\fra$ is a discrete set of rational numbers. 
As we have mentioned, $\cJ(\fra^0)=\cO_X$. The smallest jumping number is thus the smallest
$\lambda$ such that $\cJ(\fra^{\lambda})\neq\cO_X$. This is the
\emph{log canonical threshold} $\lct(\fra)$ of $\fra$. 

A result due to Ein and Lazarsfeld allows to reduce studying arbitrary multiplier ideals to those 
for which the exponent is
less than the dimension of $X$. This is Skoda's theorem (see \cite[\S 11.1.A]{positivity}), saying that if $\fra$ is locally generated by $r$ elements, then
$$\cJ(\fra^{\lambda})=\fra\cdot\cJ(\fra^{\lambda-1})\,\,\text{for}\,\,\lambda\geq r.$$
In particular, one can take $r=\dim(X)$. 

\subsection{Test ideals}
The (generalized) test ideals have been introduced by Hara and Yoshida in \cite{HY},
using a generalization of tight closure theory. Since we will only work on regular schemes, it is more convenient to use the alternative definition from 
\cite{BMS2}, that we now present. For proofs and more details, we refer to \emph{loc. cit}.

Let $X$ be a regular scheme of positive characteristic $p$. We denote by
$F\colon X\to X$ the absolute Frobenius morphism, which is the identity on the topological space, and given by the $p$-th power map on the sections of $\cO_X$. 
We assume that $X$ is $F$-\emph{finite}, that is, $F$ is a finite map
(note that $F$ is also flat since $R$ is regular).
This is satisfied, for example, if $X$ is a scheme of finite type over a perfect field,
a local ring of such a scheme, or the completion of such a ring.

Suppose first, for simplicity, that $X=\Spec R$, where $R$ is a regular $F$-finite domain. 
For an ideal $J$ in $R$ and for $e\geq 1$, one denotes by $J^{[p^e]}$
the ideal $(h^{p^e}\mid h\in J)$. Using the fact that $F$ is finite and flat one shows that given any ideal $\frb$ in $R$, there is a unique ideal $J$ such that $\frb\subseteq J^{[p^e]}$. We denote
this $J$ by $\frb^{[1/p^e]}$.

Suppose now that $\fra$ is a nonzero ideal in $R$ and $\lambda\in\RR_{\geq 0}$. 
It is not hard to see that for every $e\geq 1$ we have an inclusion
$$\left(\fra^{\lceil \lambda p^e\rceil}\right)^{[1/p^e]}
\subseteq\left(\fra^{\lceil\lambda p^{e+1}\rceil}\right)^{[1/p^{e+1}]}.$$
It follows from the Noetherian property that for $e\gg 0$ the ideal
$\left(\fra^{\lceil \lambda p^e\rceil}\right)^{[1/p^e]}$ is independent of $e$. This is the (generalized) 
\emph{test ideal} $\tau(\fra^{\lambda})$ of $\fra$ of exponent $\lambda$.
One can show that the construction of test ideals commutes with localization and completion.
In particular, we can extend the above definition to the general case when $X$ is a
regular $F$-finite scheme of positive characteristic and $\fra$ is an everywhere nonzero ideal;
 the test ideals $\tau(\fra^{\lambda})$
are coherent ideals of $\cO_X$.

The formal properties that we discussed for multiplier ideals also hold in this setting.
If $\lambda<\mu$, then $\tau(\fra^{\mu})\subseteq\tau(\fra^{\lambda})$. 
With a little effort (see \cite[Proposition~2.14]{BMS2}) one shows that for every
$\lambda\in\RR_{\geq 0}$ there is $\epsilon>0$ such that
$\tau(\fra^{\lambda})=\tau(\fra^{\mu})$ for every $\mu$ with 
$\lambda\leq\mu\leq\lambda+\epsilon$. 
A positive $\lambda$ is an $F$-\emph{jumping number} of $\fra$ if
$\tau(\fra^{\lambda})\neq\tau(\fra^{\mu})$ for every $\mu<\lambda$. 
The set of $F$-jumping numbers of $\fra$ is known to be a discrete set of rational numbers
when $X$ is essentially of finite type over a field (\cite[Theorem~3.1]{BMS2}) or
when $\fra$ is locally principal (\cite[Theorem~1.1]{BMS1}).
Note, however, that this
assertion is considerably more subtle that the corresponding one in characteristic zero.
One property that is special to characteristic $p$ says that if  $\lambda$ is an $F$-jumping number, then $p\lambda$ is an $F$-jumping number, too. 
 It follows from definition that $\tau(\fra^0)=\cO_X$, hence the first $F$-jumping number is the smallest $\lambda$ such that $\tau(\fra^{\lambda})\neq \cO_X$. This is the
$F$-\emph{pure threshold} $\fpt(\fra)$. 
We note that if $X=U_1\cup\ldots\cup U_m$ is an open cover, then
$\lambda$ is an $F$-jumping number of $\fra$ if and only if it is a jumping number
of one of the restrictions $\fra\vert_{U_i}$.

There is a version of Skoda's theorem also in this setting, and this is in fact more elementary than in the case of multiplier ideals (for a proof involving the above definition, see 
\cite[Proposition~2.25]{BMS2}). This says that if $\fra$ is locally generated by $r$ elements, then
$$\tau(\fra^{\lambda})=\fra\cdot\tau(\fra^{\lambda-1})\,\,\text{for}\,\,\lambda\geq r.$$
In particular, one can always take $r=\dim(X)$.

We end by mentioning a formula for computing ideals of the form $\frb^{[1/p^e]}$,
which we will use to compute examples in \S 6. Suppose that $X=\Spec R$, where
$R$ is a regular domain of characteristic $p>0$, such that $R$ has a basis over
$R^{p^e}$ given by $u_1,\ldots, u_r$. If the ideal $\frb$ in $R$ is generated by 
$h_1,\ldots,h_m$, and if we write
$h_i=\sum_{j=1}^ra_{i,j}^{p^e}u_j$, then
\begin{equation}\label{formula_test_ideal}
\frb^{[1/p^e]}=(a_{i,j}\mid 1\leq i\leq m, 1\leq j\leq r). 
\end{equation}
For a proof, see
\cite[Proposition~2.5]{BMS2}. We will apply this when $R=k\llbracket x_1,\ldots,x_n\rrbracket$,
with $k$ a perfect field, when we may consider the basis given by all monomials
$x_1^{i_1}\cdots x_n^{i_n}$ with $0\leq i_1,\ldots,i_n\leq p^e-1$.

\subsection{Relating multiplier ideals and test ideals via reduction mod $p$}
We start by briefly discussing the framework for reduction to positive characteristic. For details, we refer to \cite[\S 2.2]{MS}.
Suppose  that $Y$ 
is a scheme of finite type over an algebraically closed field $k$ of characteristic zero. We can find a finitely generated $\ZZ$-subalgebra
$A\subseteq k$, a scheme $Y_A$ of finite type over $A$ (a \emph{model} for $Y$), and an isomorphism
 $\phi_A\colon Y_A\times_{\Spec A}\Spec k
\to Y$. If we choose a different $B\subseteq k$, a corresponding scheme $Y_B$ over $B$, 
and an isomorphism $\phi_B\colon Y_B\times_{\Spec B}\Spec\,k\to Y$, then we can find a 
finitely generated $\ZZ$-subalgebra $C\subseteq k$ containing both $A$ and $B$
and an isomorphism $\psi\colon Y_A\times_{\Spec A}\Spec C\to
Y_B\times_{\Spec B}\Spec C$ such that 
$\psi\times_{\Spec C}\Spec k=\phi_B^{-1}\circ\phi_A$. 

Given $A$ and $Y_A$ as above, we consider closed points $s\in\Spec A$. Note that 
the residue field $k(s)$ of $s$ is finite. We denote by $Y_s$ the fiber of $Y_A$ over $s$. 

We always choose $A$ and $Y_A$ as above, but 
all properties that we will discuss refer to 
closed points in some open subset of $\Spec A$; in particular, they are independent of the choice of $A$ and $Y_A$. In light of this, we allow replacing $A$ by some localization 
$A_a$, with $a\in A$ nonzero, and $Y_A$ by $Y_A\times_{\Spec A}\Spec A_a$. 
For example, after possibly replacing $A$ by $A_a$, we may assume that $Y_A$ is 
flat over $A$. Furthermore, 
 if $Y$ is smooth (and irreducible), then 
we may assume that for every closed point $s\in \Spec A$, the fiber $Y_s$ is smooth (and irreducible, of the same dimension as $Y$). 

Given a coherent sheaf $\cF$ on $Y$, we may choose $A$ and a model $Y_A$ such that there is a sheaf
$\cF_A$ on $Y_A$ (a \emph{model} for $\cF$) whose pull-back to $Y$ is isomorphic to $\cF$. 
In this case we denote by $\cF_s$ the restriction of $\cF_A$ to $Y_s$. 
Furthermore, given a morphism of coherent sheaves $\alpha\colon \cF\to\cG$, we may choose 
models $Y_A$, $\cF_A$, and $\cG_A$ such that there is a morphism of sheaves
$\alpha_A\colon\cF_A\to\cG_A$ inducing $\alpha$. In particular, we may consider
$\alpha_s\colon\cF_s\to\cG_s$ for every closed point $s\in\Spec A$. 
Given an exact sequence of sheaves $\cF'\to\cF\to\cF''$, we may assume 
after replacing $A$ by a suitable localization $A_a$ that the sequences
$\cF'_s\to\cF_s\to\cF''_s$ are exact for all closed points $s\in\Spec A$. 
In particular, if
$\cF$ is an ideal in $\cO_Y$, then we may assume that each $\cF_s$ is an ideal in $\cO_{Y_s}$.

Given a morphism of schemes $\pi\colon Y\to X$ of finite type over $k$, we may choose $A$ and the models $Y_A$ 
and $X_A$ such that there is a morphism $\pi_A\colon Y_A\to X_A$ of schemes over $A$ inducing $\pi$. 
In this case we obtain morphisms $\pi_s\colon Y_s\to X_s$ for all closed points $s\in \Spec A$.
Furthermore, if $\pi$ is projective (or birational, finite, open or closed immersion), we may assume 
that each $\pi_s$ has the same property. Given, in addition, a coherent sheaf $\cF$ on $Y$, we may assume, after restricting to a suitable open subset of $\Spec A$ that for all $s$ we have canonical isomorphisms
\begin{equation}\label{base_change}
R^i\pi_*(\cF)_s\simeq R^i(\pi_s)_*(\cF_s).
\end{equation}

We now describe the setting that we will be interested in. 
Suppose that $\fra$ is an everywhere nonzero ideal on 
the smooth scheme $X$ over $k$.
We fix a log resolution 
$\pi\colon Y\to X$
of $\fra$, and write $\fra\cdot\cO_Y=\cO_Y(-D)$, and
$$D=\sum_{i=1}^Na_iE_i\,\,\text{and}\,\,K_{Y/X}=\sum_{i=1}^Nk_iE_i.$$
We choose $A$ and models $\pi_A$, $X_A$, $Y_A$, $D_A$, $(E_i)_A$, and $\fra_A$ such that
for every closed point $s\in\Spec A$ the induced map $\pi_s\colon Y_s\to X_s$
is a log resolution of $\fra_s$, and we have $\fra_s\cdot\cO_{Y_s}=\cO_{Y_s}(-D_s)$ and
$$D_s=\sum_{i=1}^Na_i (E_i)_s\,\,\text{and}\,\,K_{Y_s/X_s}=\sum_{i=1}^Nk_i (E_i)_s.$$
Moreover, given $\lambda\in\RR_{\geq 0}$, we may consider $\cJ(\fra^{\lambda})_s$, and it follows from (\ref{base_change}) that we may assume that
\begin{equation}\label{reduction_multiplier}
\cJ(\fra^{\lambda})_s=(\pi_s)_*\cO_{Y_s}(K_{Y_s/X_s}-\lfloor \lambda D_s\rfloor).
\end{equation}
If we consider all $\lambda$ in some bounded interval, we have the above formula
for all such $\lambda$, due to the fact that we only need to consider finitely many ideals
(it is enough to only consider those $\lambda$ such that $\lambda a_i$ is an integer 
for some $i$). If we want to consider all multiplier ideals $\cJ(\fra^{\lambda})$ and their
reductions to prime characteristic, we simply decree, motivated by Skoda's theorem, that
$\cJ(\fra^{\lambda})_s=\fra_s\cdot \cJ(\fra^{\lambda-1})_s$ for $\lambda\geq\dim(X_s)$;
this reduces us to only having to define $\cJ(\fra^{\lambda})_s$ for $\lambda<\dim(X)$.
In what follows we simply refer to all the above choices as a \emph{model for the multiplier ideals
of} $\fra$.

We can now formulate the two main results due to Hara and Yoshida 
concerning the connection between the reductions of multiplier ideals and the corresponding
test ideals (see \cite[Theorem~3.4]{HY}). We assume that $X$ is a smooth scheme over $k$,
and $\fra$ is an everywhere nonzero ideal on $X$; furthermore, we choose models for the multiplier ideals
of $\fra$ over some finitely generated $\ZZ$-algebra $A\subset k$. 

\begin{theorem}\label{thm1_HY}
With the above notation, after possibly replacing $A$ by a localization $A_a$, we have
$$\tau(\fra_s^{\lambda})\subseteq \cJ(\fra^{\lambda})_s$$
for all closed points $s\in\Spec A$ and all $\lambda\in\RR_{\geq 0}$.
\end{theorem}

Note that even if we start with models for $X$ and $\fra$, in order to apply
Theorem~\ref{thm1_HY} we might need to change $A$.
Indeed, we need to guarantee that some log resolution of $\fra$ is defined over $A$, and that
(\ref{reduction_multiplier}) holds for $\lambda<\dim(X)$. 

\begin{theorem}\label{thm2_HY}
With the above notation, given any $\lambda\in\RR_{\geq 0}$, there is an open subset
$V_{\lambda}\subseteq\Spec A$ such that 
$$\tau(\fra_s^{\lambda})=\cJ(\fra^{\lambda})_s$$
for all closed points $s\in V_{\lambda}$.
\end{theorem}

It is definitely not the case that $V_{\lambda}$ can be taken independently of $\lambda$.
However, it is expected that there is a dense set of closed points in $s\in \Spec A$
such that the equality in Theorem~\ref{thm2_HY} holds for all $\lambda\in\RR_{\geq 0}$,
see \cite{MS}.

The proof of Theorem~\ref{thm1_HY} is elementary. For a proof in our simplified setting,
and with our definitions, see \cite[Proposition~4.3]{BHMM}. The proof of Theorem~\ref{thm2_HY}
is deeper, and makes use of the action of the Frobenius on the de Rham complex. 
We will make use of the main ingredient 
in this proof in \S 4, in order to give lower
bounds for the $F$-jumping numbers of the reductions of $\fra$ to positive characteristic.

In particular, the above two theorems give the following relation between the log canonical
threshold of $\fra$ and the $F$-pure thresholds of the reductions $\fra_s$ to positive characteristic. If $A$ is chosen as in Theorem~\ref{thm2_HY}, then the theorem implies that
$\lct(\fra)\geq\fpt(\fra_s)$ for all closed points $s\in\Spec A$. On the other hand, 
Theorem~\ref{thm2_HY} implies that for every $\epsilon>0$ there is an open subset
$U_{\epsilon}\subseteq \Spec A$ such that $\lct(\fra)-\fpt(\fra_s)<\epsilon$ for every closed point $s\in U_{\epsilon}$.

The above results raise the following problem that we will consider in the next two sections.
Suppose that $\lambda$ is a jumping number of $\fra$, and denote by $\lambda'$ the largest
jumping number $<\lambda$ (when $\lambda=\lct(\fra)$, we put $\lambda'=0$).
After replacing $A$ by some localization $A_a$, we may assume by
Theorem~\ref{thm2_HY} that for every closed point $s\in\Spec A$ we have 
$$\cJ(\fra^{\lambda'})_s=\tau(\fra_s^{\lambda'})\,\,\text{and}\,\,
\cJ(\fra^{\lambda})_s=\tau(\fra_s^{\lambda}).$$
It follows that in this case there is an $F$-jumping number of $\fra_s$ in the interval
$(\lambda',\lambda]$. The following problem addresses the question of how close this jumping number is from $\lambda$.

\begin{problem}\label{pb1}
With the above notation, show that there are $C\in\RR_{>0}$ and $N\in\ZZ_{>0}$ 
such that 
after possibly replacing $A$ by a localization, for every closed point in
$\Spec A$ and every 
$F$-jumping number $\mu$ of $\fra_s$ in the open interval $(\lambda', \lambda)$,
we have
\begin{equation}\label{eq_problem}
\frac{1}{{\rm char}(k(s))^{N}}\leq \lambda-\mu\leq\frac{C}{{\rm char}(k(s))}.
\end{equation}
\end{problem}

In the next section we prove the existence of $N$ in Problem~\ref{pb1}
when $\fra$ is a locally principal ideal, while in \S 4 we show how to find $C$ for arbitrary $\fra$.

\section{Upper bounds for $F$-jumping numbers}

The key ingredient for giving an upper bound for $F$-jumping numbers is provided by the following result, which is \cite[Proposition~4.3]{BMS1}. We include the proof for the sake of completeness.

\begin{lemma}\label{lem_BMS}
Let $X$ be a regular $F$-finite scheme of characteristic $p>0$, and let 
$\fra$ be an everywhere nonzero locally principal ideal in $\cO_X$. 
Given $\lambda=\frac{r}{p^e-1}$ for positive integers $r$ and $e$, let us put
$\lambda_m=\left(1-\frac{1}{p^{me}}\right)\lambda$ for $m\geq 0$. If there is an 
$F$-jumping number of $\fra$ in $(\lambda_m,\lambda_{m+1}]$ with $m\geq 1$, then
there is also an $F$-jumping number of $\fra$ in $(\lambda_{m-1},\lambda_m]$. In particular,
in this case there are at least $m+1$ $F$-jumping numbers of $\fra$ in $(0,\lambda)$.
\end{lemma}

\begin{proof}
Note that Skoda's theorem for a locally principal ideal implies that $\mu>1$ is an $F$-jumping number if and only if $\mu-1$ is an $F$-jumping number. Furthermore, as we have mentioned, if $\mu$ is an $F$-jumping number for $\fra$, then also $p\mu$ is an
$F$-jumping number. It is easy to see that if $\mu\in (\lambda_m,\lambda_{m+1}]$, then
$p^e\mu-r\in (\lambda_{m-1},\lambda_m]$, and as we have seen, $p^e\mu-r$
is an $F$-jumping number of $\fra$.
\end{proof}

\begin{corollary}\label{cor_BMS}
With $X$ and $\fra$ as in Lemma~\ref{lem_BMS}, if there are at most $d$ $F$-jumping numbers
of $\fra$ that are $<\lambda=\frac{r}{p^e-1}$, then there are no $F$-jumping numbers of $\fra$ in the open interval $(\lambda_d,\lambda)$.
\end{corollary}

We can now state and prove the main result of this section. Suppose that 
$A$ is a finitely generated $\ZZ$-algebra, and $X$ is a smooth scheme of finite type over 
$\Spec A$, of relative dimension $n$. Let $\fra$ be a locally principal ideal on $X$,
whose restriction to every fiber is everywhere nonzero.

\begin{theorem}\label{thm_main1}
With the above notation, if $\lambda\in\QQ_{>0}$, then 
there are positive integers $N$ and $p_0=p_0(\lambda)$ such that for every closed point $s\in\Spec A$ 
with ${\rm char}(k(s))>p_0$ and every $F$-jumping number
$\mu$ of $\fra_s$ that is $<\lambda$ we have
$$\lambda-\mu\geq\frac{1}{{\rm char}(k(s))^{N}}.$$
In particular, this holds for every closed point $s\in\Spec A_a$, for some nonzero
$a\in A$. 
\end{theorem}

\begin{proof}
After taking a suitable affine open cover of $X$, we may assume that 
$X={\rm Spec}(R)$ is affine and
$\fra=(f)$ is a principal ideal.
Let us write $\lambda=\frac{a}{b}$, with $a,b\in\ZZ_{>0}$. We first require that
$p_0$ is such that all divisors of $b$ are $\leq p_0$.

\noindent{\bf Claim}. There is a positive integer $d$ such that for every closed point
$s\in\Spec A$, there are at most $d$ $F$-jumping numbers of $f_s$ that are
$<\lambda$. 

Indeed, since $X$ is of finite type over $A$, we can write
$R\simeq A[x_1,\ldots,x_m]/J$
for some ideal $J$, and let $g\in A[x_1,\ldots,x_m]$ be a polynomial whose class 
corresponds to $f$.  
It follows from
\cite[Proposition~3.6]{BMS2} that $\tau(f_s^{\mu})=\tau((J_s+(g_s))^{\mu+m-n})\cdot R_s$.
On the other hand, let $M$ be such that $J+(g)\subseteq A[x_1,\ldots,x_m]$ is 
generated in degree $\leq M$. It follows from \cite[Proposition~3.2]{BMS2} that
$\tau((J_s+(g_s))^{\mu+m-n})$ is generated in degree $\leq \lfloor M(\mu+m-n)\rfloor$
for every $\mu$. In particular, the number of $F$-jumping numbers of $f_s$ that are $<\lambda$
is bounded above by the number of $F$-jumping numbers of $J_s+(g_s)$ that are
$<\lambda+m-n$, which in turn is bounded above by the dimension of the vector space of polynomials in $k(s)[x_1,\ldots,x_m]$ of degree $\leq \lfloor M(\lambda+m-n)\rfloor$. 
Note that this dimension is independent of the closed point $s\in\Spec A$, hence we obtain $d$
as in the claim.

Let us fix $d$ as in the claim, and let $s\in\Spec A$ be a closed point. Let 
$p={\rm char}(k(s))$ and let $e\geq 1$ be the
order of $p$ in the group of units of $\ZZ/b\ZZ$ (recall that $p$ does not divide $b$). 
In this case we can write $\lambda=\frac{r}{p^e-1}$, and it follows from Corollary~\ref{cor_BMS}
that if $\mu<\lambda$ is an $F$-jumping number of $\fra_s$, then 
$$\lambda-\mu\geq\lambda-\lambda_d=\frac{\lambda}{p^{de}}.$$
Note that $e\leq b$, and if we require also $p_0\geq\lambda^{-1}$, we see that
$N=db+1$ satisfies the first assertion in the theorem. 

The second assertion is clear, too: if we consider
$u\colon \Spec A\to\Spec \ZZ$, it is is enough to take a nonzero $a\in A$
such that $u(\Spec A_a)$ does not contain any prime $p\ZZ$, with $0<p\leq p_0$. 
\end{proof}

\begin{corollary}\label{thm_main1_v2}
Given $\lambda\in\QQ_{>0}$ and $n,M\in\ZZ_{>0}$, there are positive integers $N=N(n,M,\lambda)$
and $p_0=p_0(\lambda)$ 
such that for every $F$-finite field $k$ of  characteristic $p\geq p_0$, and every
$f\in k[x_1,\ldots,x_n]$ with ${\rm deg}(f)\leq M$, we have
$$\lambda-\mu\geq\frac{1}{p^N}$$
for every $F$-jumping number $\mu<\lambda$ of $f$.
\end{corollary}

\begin{proof}
We could either apply Theorem~\ref{thm_main1} for the universal polynomial of degree
$\leq M$ in $n$ variables (with $\Spec A$ being the parameter space for such polynomials),
or simply apply the argument in the proof of the theorem, noting that in this case
we obtain directly the bound for the number of $F$-jumping numbers of $f$ that are $<\lambda$,
in terms of $n$, $M$, and $\lambda$. 
\end{proof}

Suppose now that $X$ is a smooth scheme over an algebraically closed field 
$k$, and $\fra$ is an everywhere nonzero, locally principal ideal on $X$. Let us consider models for $X$ and
 $\fra$
over a finitely generated $\ZZ$-algebra $A\subset k$. 
By applying Theorem~\ref{thm_main1} to $\lambda=\lct(\fra)$, we obtain the following.

\begin{corollary}\label{cor_lct}
With the above notation, there is $N$ such that after possibly replacing $A$ by a localization
$A_a$, we have $\lct(\fra)-\fpt(\fra_s)\geq \frac{1}{{\rm char}(k(s))^N}$ for every closed point
$s\in\Spec A$ such that $\lct(\fra)\neq\fpt(\fra_s)$.
\end{corollary}

\section{Lower bounds for $F$-jumping numbers}

In this section we assume that $X$ is a smooth 
scheme of finite type over an algebraically closed field $k$ of characteristic zero. We consider 
an everywhere nonzero ideal sheaf $\fra$ in $\cO_X$ and
fix a model over a finitely generated $\ZZ$-algebra $A\subset k$ for $X$, $\fra$, and the multiplier ideals of $\fra$. Our goal in this section is to prove the following theorem.

\begin{theorem}\label{thm_main2}
With the above notation, for every $\lambda>0$  there is 
$C\in\RR_{>0}$ such that after possibly replacing $A$ by a localization $A_a$,
for every closed point $s\in\Spec A$ we have
\begin{equation}\label{eq1_thm_main2}
\cJ(\fra^{\lambda-\frac{C}{p_s}})_s=\tau(\fra_s^{\lambda-\frac{C}{p_s}}),
\end{equation}
where $p_s={\rm char}(k(s))$. 
\end{theorem}

\begin{remark}\label{rem0_thm_main2}
Note that the interesting inclusion in (\ref{eq1_thm_main2}) in ``$\subseteq$",
since the reverse one can be guaranteed using Theorem~\ref{thm1_HY}.
\end{remark}

\begin{remark}\label{rem_thm_main2}
With the notation in the above theorem, we may assume that for every closed point $s\in
\Spec A$, we have $\lambda-\frac{C}{p_s}\geq\lambda'$, where $\lambda'$ 
is the largest jumping number of $\fra$ that is $<\lambda$ (with the convention that
$\lambda'=0$ if there is no such jumping number). In this case, the ideal on the left-hand side of
(\ref{eq1_thm_main2})
is $\cJ(\fra^{\lambda'})_s$. 
\end{remark}

\begin{corollary}\label{cor_thm_main2}
With the notation in Remark~\ref{rem_thm_main2}, if $\lambda$ is a jumping number of $\fra$,
then there is $C>0$ such that after possibly replacing $A$ by a localization $A_a$,
the following holds: 
for every closed point $s\in\Spec A$,
there is an $F$-jumping number $\mu\in (\lambda',\lambda]$ for $\fra_s$,
and for every such $\mu$ we have
$$\lambda-\mu\leq\frac{C}{p_s}.$$
\end{corollary}

\begin{proof}
It follows from 
Theorem~\ref{thm1_HY} that we may assume
$\tau(\fra_s^{\alpha})\subseteq\cJ(\fra^{\alpha})_s$ for every $\alpha\in\RR_{>0}$. 
Suppose now that the conclusion of Theorem~\ref{thm_main2} holds.
In this case, for every closed point $s\in\Spec A$ and every 
$\alpha$ with
 $\lambda'\leq \alpha\leq\lambda-\frac{C}{p_s}$, we have
 \begin{equation}\label{eq_cor_thm_main2}
 \cJ(\fra^{\lambda-\frac{C}{p_s}})=\tau(\fra_s^{\lambda-\frac{C}{p_s}})\subseteq
 \tau(\fra_s^{\alpha})\subseteq\tau(\fra_s^{\lambda'})\subseteq\cJ(\fra^{\lambda'})_s.
 \end{equation}
 Since $\cJ(\fra^{\lambda-\frac{C}{p_s}})=\cJ(\fra^{\lambda'})_s$, we conclude that
all inclusions in (\ref{eq_cor_thm_main2}) are equalities. In particular, there is no 
$F$-jumping number for $\fra_s$ in $(\lambda',\lambda-\frac{C}{p_s}]$.

On the other hand, since $\lambda$ is a jumping number for $\fra$, we have
$$\tau(\fra_s^{\lambda})\subseteq \cJ(\fra^{\lambda})_s\subsetneq\cJ(\fra^{\lambda-\frac{C}{p_s}})_s=\tau(\fra^{\lambda-\frac{C}{p_s}}_s).$$ Therefore there is an $F$-jumping number of $\fra_s$ in the interval
$(\lambda-\frac{C}{p_s},\lambda]$, and we thus obtain both assertions in the corollary. 
\end{proof}

\begin{corollary}\label{cor2_thm_main2}
With the notation in Theorem~\ref{thm_main2}, there is $C>0$ such that after possibly replacing $A$ by a localization $A_a$, we have $\lct(\fra)-\fpt(\fra_s)<\frac{C}{p_s}$
for every closed point $s\in\Spec A$.
\end{corollary}

\begin{proof}
Applying Theorem~\ref{thm_main2} with $\lambda=\lct(\fra)$, we see that we may assume that
$\tau(\fra_s^{\lambda-\frac{C}{p_s}})=\cO_{X_s}$, hence $\fpt(\fra_s)>\lambda-\frac{C}{p_s}$.
\end{proof}

\begin{remark}\label{rem2_thm_main2}
The assertion in Theorem~\ref{thm_main2} is only interesting when
$\lambda$ is a jumping number of $\fra$.
Indeed, otherwise we can find $\epsilon>0$ such that 
$\cJ(\fra^{\lambda})=\cJ(\fra^{\lambda-\epsilon})$. By applying Theorem~\ref{thm2_HY}, we see that we may assume that for all closed points $s\in\Spec A$ we
 have 
 $$\tau(\fra_s^{\lambda-\epsilon})=\cJ(\fra^{\lambda-\epsilon})_s=\cJ(\fra^{\lambda})_s=
 \tau(\fra_s^{\lambda}).$$
 If we consider a localization $A_a$ of $A$ such that for every closed point $s\in\Spec A$
 we have ${\rm char}(k(s))>\frac{1}{\epsilon}$, it is clear that the conclusion of the theorem holds 
by taking $C=1$.
\end{remark}

\begin{remark}\label{rem_reduction}
In order to prove Theorem~\ref{thm_main2}, we may assume that $X$ is affine and irreducible, and $\fra$ is a principal ideal.
Indeed, after taking an open affine cover, we reduce to the case when $X$ is affine and irreducible.
Consider now generators $g_1,\ldots,g_m$ for the ideal $\fra$. If $M>\lambda$ is an integer,
and $h_1,\ldots,h_M$ are general linear combinations of the $g_i$ with coefficients in $k$,
then 
$$\cJ(\fra^{\alpha})=\cJ(h^{\alpha/M})$$
for every $\alpha<M$, where $h=h_1\cdots h_M$
(see \cite[Proposition~9.2.28]{positivity}). If the theorem holds for $h$ and $\lambda/M$, then we 
can find $C'>0$ such that after replacing $A$ by a localization
$$\cJ(h^{\frac{\lambda}{M}-\frac{C'}{p_s}})_s=\tau(h_s^{\frac{\lambda}{M}-\frac{C'}{p_s}})$$
for every closed point $s\in \Spec A$. Using the fact that $h\in\fra^M$, we now obtain
$$\cJ(\fra^{\lambda-\frac{C'M}{p_s}})_s=\cJ(h^{\frac{\lambda}{M}-\frac{C'}{p_s}})_s
=\tau(h_s^{\frac{\lambda}{M}-\frac{C'}{p_s}})
\subseteq\tau(\fra_s^{\lambda-\frac{C'M}{p_s}}),$$
hence we obtain the inclusion ``$\subseteq$" in Theorem~\ref{thm_main2} if we take $C=C'M$,
while the reverse inclusion is trivial (see Remark~\ref{rem0_thm_main2}).
\end{remark}

Before giving the proof of Theorem~\ref{thm_main2}, we recall the criterion from \cite{HY} that guarantees the equality of multiplier ideals and test ideals in a fixed positive characteristic. 
This is the heart of the proof of Theorem~\ref{thm2_HY}. We start by describing the setting. 

Suppose that $X'$ is a smooth, irreducible, $n$-dimensional affine scheme over a perfect 
field $L$ of characteristic $p>0$. We assume that we have a log resolution $\pi'\colon Y'\to X'$ of  a nonzero
principal ideal $\fra'$ on $X'$. Let $Z'$ be the divisor on $Y'$ such that $\fra'\cdot\cO_{Y'}=
\cO_{Y'}(-Z')$. By assumption, there is a simple normal crossing divisor
$E'=E'_1+\ldots+E'_N$ such that both $K_{Y'/X'}$ and $Z'$ are supported on $E'$. 
Under these assumptions, we put 
$\cJ({\fra'}^{\alpha})=\pi_*\cO_{Y'}(K_{Y'/X'}-\lfloor\alpha Z'\rfloor)$.
It is shown in \cite{HY} that in this case we have $\tau({\fra'}^{\alpha})\subseteq
\cJ({\fra'}^{\alpha})$ for every $\alpha$ (this is the result that implies Theorem~\ref{thm1_HY}).

Suppose now that $\alpha\in\RR_{\geq 0}$ is fixed, and we choose $\mu>\alpha$ such that
$\cJ({\fra'}^{\alpha})=\cJ({\fra'}^{\mu})$ (note that if $Z'=\sum_ia_iZ'_i$, then it is enough to take
$\mu<\frac{\lfloor \alpha a_i\rfloor +1-\alpha a_i}{a_i}$ for all $i$ with $a_i>0$). 
Suppose also that we have a $\QQ$-divisor $D'$ on $Y'$ supported on $E'$ such that 
$D'$ is ample over $X'$ and $-D'$ is effective.
We put $G'=\mu(D'-Z')$ and assume, in addition,  
that $\lceil G'\rceil=\lceil-\mu Z'\rceil$. 
The following is the main criterion for the equality
of $\tau({\fra'}^{\alpha})$ and $\cJ({\fra'}^{\alpha})$. We denote by 
$\Omega_{Y'}^i({\rm log}\,E')$ the sheaf of $i$-differential forms on $Y'$ with log poles along $E'$. If $F'=\sum_{i=1}^N\alpha_iE'_i$ is a divisor with real coefficients, then we put
$\lceil F\rceil=\sum_{i=1}^N\lceil \alpha_i\rceil E'_i$, where $\lceil u\rceil$ denotes the smallest
integer $\geq u$.

\begin{proposition}\label{criterion}
With the above notation, if the following conditions hold:
\begin{enumerate}
\item[(A)] $H^i(Y',\Omega^{n-i}_{Y'}({\rm log}\,E')(-E'+\lceil p^{\ell} G'\rceil))=0$ for all $i\geq 1$ and $\ell\geq 1$;
\item[(B)] $H^{i+1}(Y', \Omega^{n-i}_{Y'}({\rm log}\,E')(-E'+\lceil p^{\ell} G'\rceil))=0$ for all $i\geq 1$ and $\ell\geq 0$,
\end{enumerate}
then $\tau({\fra'}^{\alpha})=\cJ({\fra'}^{\alpha})$. 
\end{proposition}

We refer to \cite{HY} for a proof of this result. For a somewhat simplified version of the argument,
using our definition of test ideals, see the presentation in \cite[\S 4]{BHMM}.
We can now prove the main result of this section.

\begin{proof}[Proof of Theorem~\ref{thm_main2}]
It is easy to see that the assertion in the theorem is independent of the models we have chosen.
The point is that if $A\subset B$ is an inclusion of finitely generated $\ZZ$-algebras,
and $s\in\Spec A$ is the image of the closed point $t\in\Spec B$, then we have a finite field extension
$k(s)\subseteq k(t)$; if we have models of $X$ and $\fra$ over both $A$ and $B$ such that
$X_B\simeq X_A\times_{\Spec A}\Spec B$, then it
follows from the description of test ideals in \cite[Proposition~2.5]{BMS2} that
$\tau(\fra_t^{\mu})=\tau(\fra_s^{\mu})\otimes_{k(s)}k(t)$.

In particular, we may change the log resolution, and we may assume that all divisors that we define in characteristic zero have models over $A$. We also note that by
Remark~\ref{rem2_thm_main2}, we may assume that $\lambda$
is a jumping number of $\fra$, hence it is rational. Finally, by Remark~\ref{rem_reduction},
we may and will assume that $X$ is affine and irreducible, and $\fra$ is principal.

By assumption, we have a log resolution $\pi\colon Y\to X$ 
of $\fra$ that admits a model over $A$,
such that construction of multiplier ideals commutes with taking the fiber over the closed points 
in $\Spec A$. Since we may choose a convenient log resolution, we may assume that
$\pi$ is an isomorphism over $X\smallsetminus V(\fra)$.
Let $Z$ be the divisor on $Y$ such that $\fra\cdot\cO_{Y}=\cO_{Y}(-Z)$.
By assumption, $Z$ is supported on a
simple normal crossing divisor $E=E_1+\ldots+E_N$ on $Y$, 
such that the support of $K_{Y/X}$ is also contained in $E$.  

We fix an integral divisor $H$ on $Y$ that is ample over $X$ and such that 
$-H$ is effective and supported on $E$
(for example, if we express $Y$
as the blow-up of $X$ along some closed subscheme $T$, then 
we may choose $H=-\pi^{-1}(T)$). 
Let us write 
$$Z=\sum_{i=1}^Na_iE_i\,\,\text{and}\,\,-H=\sum_{i=1}^Nh_iE_i,$$
so by assumption $a_i>0$ and $h_i\geq 0$ for all $i$. 

\begin{lemma}\label{lem1_thm_main2}
With the above notation, if $\cF$ is a coherent sheaf on $Y$, there is $m_0$ such that
after possibly replacing $A$ by a localization $A_a$, we have
$$H^i(Y_s,\cF_s\otimes \cO_{Y_s}(mH_s))=0$$
for all $i\geq 1$, all $m\geq m_0$, and all closed points $s\in\Spec A$. 
\end{lemma}

\begin{proof}
We may assume that $\cO_Y(H)$ is very ample over $X$:
indeed, there is $d\geq 1$ such that $\cO_Y(dH)$ is very ample over $X$,
and it is enough to prove the assertion in the lemma for $\cO_Y(dH)$ and each of the sheaves
$\cF\otimes\cO_Y(jH)$, with $0\leq j\leq d-1$. We now assume that $\cO_Y(H)$ is very ample over $X$. By asymptotic Serre vanishing, there is $m_0$ such that 
$H^i(Y,\cF\otimes \cO_{Y}(mH))=0$ for every $i\geq 1$ and every $m\geq m_0$.
After possibly replacing $A$ by some localization $A_a$, we may assume that
$H^i(Y_s,\cF_s\otimes \cO_{Y_s}(mH_s))=0$ for every $i\geq 1$, every $m$ with
$m_0\leq m\leq m_0+n-1$, and every closed point $s\in\Spec A$.
For every such $s$, the sheaf $\cF_s$ on $Y_s$ is $(m_0+n)$-regular
with respect to $\cO_{Y_s}(H_s)$ in the sense of Castelnuovo-Mumford regularity (we refer to \cite[\S 1.8]{positivity} for the basic facts on Castelnuovo-Mumford regularity). In this case
$\cF_s$ is $m$-regular for every $m\geq m_0+n$, and we obtain all the vanishings in the statement of the lemma by definition of Castelnuovo-Mumford regularity.
\end{proof}

\begin{corollary}\label{cor5_thm_main2}
With the above notation, if $\cF$ is a coherent sheaf on $Y$, $d$ is a positive integer, 
and $\gamma$ is a positive rational number,
then there is $m_1$ such that 
$$H^i(Y_s,\cF_s(\lceil m\gamma H_s-qZ_s\rceil)=0$$
for all $m\geq \frac{m_1}{\gamma}$, all $q\in\QQ$ with $qd\in\ZZ$, all $i\geq 1$, and all closed points
$s\in\Spec A$. 
\end{corollary}

\begin{proof}
By assumption, both $m\gamma$ and $q$ have bounded denominators, hence
$m\gamma H-qZ$ can be written as $\lfloor m\gamma\rfloor H-\lfloor q\rfloor Z+T$, where when  we vary $m$ and $q$, the $\QQ$-divisor $T$ can take  only finitely many values
$T_1,\ldots,T_r$. Furthermore, note that since $\fra$ is assumed to be principal, $Z$ is the pull-back of a divisor from $X$. Since $X$ is affine, it follows that 
$H^i(Y_s,\cF_s(\lceil m\gamma H_s-qZ_s\rceil))=0$ if $H^i(Y_s,\cF_s(\lceil T\rceil+\lfloor m\gamma 
\rfloor H))=0$. Hence we obtain our assertion by applying Lemma~\ref{lem1_thm_main2}
to each of the sheaves $\cF(\lceil T_i\rceil)$, for $1\leq i\leq r$.
\end{proof}

We now return to the proof of Theorem~\ref{thm_main2}.
We first apply Corollary~\ref{cor5_thm_main2} to choose $m_1$ such that
after possibly replacing $A$ by a localization, we have 
\begin{equation}\label{eq1_pf_thm_main2}
H^i(Y_s,\Omega^j_{Y_s}({\rm log}\,E_s)(-E_s+\lceil  m\gamma H_s-qZ_s\rceil))=0
\end{equation}
for all closed points $s\in\Spec A$, all $i\geq 1$, $j\geq 0$, $m\geq m_1/\gamma$, and $q\in\QQ$ with $qd\in\ZZ$, where $\gamma=\frac{1}{3}\cdot\prod_{h_i>0}\frac{1}{h_i}$ and 
$d\in 2\ZZ$ is such that $\lambda d\in\ZZ_{>0}$.  

Let $C\in\ZZ_{>0}$ be such that $\frac{C}{3}\cdot\min\left\{\frac{a_i}{h_i}\mid h_i>0\right\}\geq m_1$. 
After possibly replacing $A$ by a localization $A_a$, we may assume that 
for every closed point $s\in\Spec A$, the characteristic $p_s$ of $k(s)$ is large enough.
In particular, we may assume that $\cJ(\fra^{\lambda-\frac{C}{p_s}})=
\cJ(\fra^{\lambda-\frac{C}{2p_s}})$ and that
\begin{equation}\label{eq2_pf_thm_main2}
\left\lfloor\left(\lambda-\frac{C}{2p_s}\right)a_i\right\rfloor=\lceil\lambda a_i\rceil -1
\,\,\text{for all}\,\,i\,\,\text{with}\,\, 1\leq i\leq N.
\end{equation}

Suppose now that $s\in\Spec A$ is a closed point. We
use primes to denote the corresponding varieties and divisors obtained after restricting
the models over $\Spec k(s)$, and write $p=p_s$. 
In order to show that the condition in the theorem is satisfied,
it is enough to show that we may apply Proposition~\ref{criterion} to 
$\alpha=\lambda-\frac{C}{p}$,
$\mu=\lambda-\frac{C}{2p}$, and $G'=\mu(\eta H'-Z')$, where $\eta\in\QQ_{>0}$ is given by
\begin{equation}\label{eq3_pf_main_thm2}
\eta=\frac{C}{3p\mu}\min\left\{\frac{a_i}{h_i}\mid h_i>0\right\}.
\end{equation}

In order to show that $\lceil G'\rceil=\lceil -\mu Z'\rceil$, it is enough to show that
\begin{equation}\label{eq4_pf_thm_main2}
-\mu\eta h_i-\mu a_i>\lceil -\mu a_i\rceil -1\,\,\text{for all}\,\,i\leq N.
\end{equation}
This is clear if $h_i=0$, hence let us assume that $h_i>0$. 
Using (\ref{eq2_pf_thm_main2}) and (\ref{eq3_pf_main_thm2}), we obtain
$$\mu\eta h_i\leq \frac{Ca_i}{3p}<\frac{Ca_i}{2p}\leq \lceil\lambda a_i\rceil -\left(\lambda-
\frac{C}{2p}\right)a_i=\lfloor\mu a_i\rfloor-\mu a_i+1,$$
which gives (\ref{eq4_pf_thm_main2}).

Arguing as in the proof of \cite[Theorem~3.4]{HY}, it is easy to guarantee the vanishing in 
(B) for $\ell=0$. Indeed, after possibly replacing $A$ by a localization, we may assume that
$$\dim_{k(s)}H^{i+1}(Y', \Omega^{n-i}_{Y'}({\rm log}\,E')(-E'+\lceil  G'\rceil))$$
is independent of the closed point $s\in\Spec A$. On the other hand, there is such a 
closed point $t\in\Spec A$
with the property that ${\rm char}(k(t))> \dim(X)$ and $Y_t$ and $E_t$ admit liftings to the second ring of Witt vectors
$W_2(k(t))$ of $k(t)$. In this case, a version of the Akizuki-Nakano vanishing theorem
(see \cite[Corollary~3.8]{Hara}) gives
$$H^{i+1}(Y_t, \Omega^{n-i}_{Y_t}({\rm log}\,E_t)(-E_t+\lceil  G_t\rceil))=0,$$
where $G=\left(\lambda-\frac{C}{2p}\right)(\eta H-Z)$. 

On the other hand, the other vanishings in (A) and (B) hold by our choice of $m_1$.
Indeed, for $\ell\geq 1$ we have 
$p^{\ell}G'=p^{\ell}\mu\eta H'-p^{\ell}\mu Z')$, and $p^{\ell}\mu\eta\geq p\mu\eta\geq m_1$ by the definition
of $C$. Moreover, $p^{\ell}\mu d=p^{\ell-1}(p\lambda d-\frac{C d}{2})\in\ZZ$ and $p^{\ell}\mu\eta\in\ZZ\gamma$.
Therefore we may apply Proposition~\ref{criterion} to obtain the assertion in the theorem.
\end{proof}

\section{Bounds for Hartshorne-Speiser-Lyubeznik numbers}

In this section we apply the results in \S 3 to get bounds for the Hartshorne-Speiser-Lyubeznik numbers. We start we recalling the definition of these numbers. 

Let $(S,\frn,k)$ be a $d$-dimensional Noetherian local ring of characteristic $p>0$. Given any $S$-module $M$, a $p$-\emph{linear structure} on $M$ is an additive map $\varphi\colon M\to M$ such that 
$\varphi(az)=a^p\varphi(z)$ for all $a\in S$ and $z\in M$. For example, the Frobenius endomorphism on $S$ gives a $p$-linear structure on $S$, and it also
induces by functoriality a $p$-linear structure $\Theta$ on the top local cohomology module $H^d_{\frn}(S)$.

We recall the following theorem, originally proved by Hartshorne-Speiser 
\cite[Proposition 1.11]{HS} and later generalized by
 Lyubeznik \cite[Proposition 4.4]{Lyubeznik} (see also
\cite{Sharp} for a simplified proof).

\begin{theorem}\label{HSL_theorem}
Let $(S,\frn,k)$ be a Noetherian local ring of characteristic $p>0$,
and let $M$ be an Artinian $S$-module with a $p$-linear structure $\varphi\colon M\to M$. If each element $z\in M$ is nilpotent under $\varphi$, then $M$ is nilpotent under $\varphi$, i.e., there exists a positive integer $\ell$ such that $\varphi^{\ell}(M)=0$.
\end{theorem}

The above theorem has the following immediate consequence.
Let $\varphi$ be a $p$-linear structure on an $S$-module $M$ and set 
$$N_i=\{z\in M\mid \varphi^i(z)=0\}.$$
When $M$ is an Artinian $S$-module, it follows from Theorem~\ref{HSL_theorem}
that the ascending chain of submodules 
$$\cdots \subseteq N_i\subseteq N_{i+1}\subseteq \cdots$$
eventually stabilizes, i.e. there is an integer $\ell$ such that $N_{\ell}=N_{\ell+j}$ for all $j\geq 1$.
This motivates the following definition.

\begin{definition}[Hartshorne-Speiser-Lyubeznik number]
With the above notation, if
$M$ is an Artinian $S$-module and $\varphi$ is a $p$-linear structure on $M$, then the  Hartshorne-Speiser-Lyubeznik number (HSL number, for short) of $(M,\varphi)$ is the smallest
positive integer $\ell$ such that $N_{\ell}=N_{\ell+j}$  for all $j\geq 1$.
\end{definition}

\begin{remark}
For every complete Noetherian local ring $T$, Matlis duality gives an inclusion-reversing one-to-one correspondence between the ideals of $T$ and the submodules of 
the injective hull $E_T$ of the residue field of $T$.
This correspondence is given by 
$$I \to {\rm Ann}_{E_T}I=\{u\in E_T\mid Iu=0\},\,\,\text{with inverse}\,\,M\to 
 {\rm Ann}_TM.$$
Therefore the stabilization of an ascending chain of submodules 
$\cdots\subseteq N_i\subseteq N_{i+1}\subseteq \cdots$ of $E_T$ is equivalent to the stabilization of $\cdots\supseteq I_i\supseteq I_{i+1}\supseteq \cdots$ where $I_i=
{\rm Ann}_TN_i$.
\end{remark}

The injective hull $E_S$ of the residue field of $S$ is of particular interest due to the following remark.

\begin{remark}
Suppose that $S$ is a complete local Noetherian ring of characteristic $p>0$, hence by Cohen's theorem $S$ is a homomorphic image of a complete regular local ring $R$, i.e. $S \simeq R/I$. In this case there is a one-to-one correspondence between $(I^{[p]}\colon I)/I^{[p]}$ and the set of $p$-linear structures 
$\varphi$ on $E_S$, as follows.

By Cohen's theorem,
we may assume that $R=k\llbracket x_1,\dots,x_n\rrbracket$, for a field $k$. 
The injective hull $E_R$ 
is isomorphic to $H^n_{\frm}(R)$, where $\frm=(x_1,\dots,x_n)$. 
The natural $p$-linear structure $F$ on $E_R$ is given by
$$\left[\frac{r}{x^{j_1}_1\cdots x^{j_n}_n}\right] \to 
\left[\frac{r^p}{x^{j_1p}_1\cdots x^{j_np}_n}\right]$$
for each cohomology class $\left[\frac{r}{x^{j_1}_1\cdots x^{j_n}_n}\right]\in 
H^n_{\frm}(R)\simeq R_{x_1\cdots x_n}/\sum_{i=1}^nR_{x_1\cdots \widehat{x_i}\cdots x_n}$.  
Given an element $u\in (I^{[p]}\colon I)$, we claim that $uF$ induces a $p$-linear structure on 
$E_S$. Note that $E_S$ can be identified with ${\rm Ann}_{E_R}I$. Given any $z\in E_S$, i.e. an element in $E_R$ that is annihilated by $I$, we have 
$$IuF(z)\subseteq I^{[p]}F(z)=F(Iz)=F(0)=0,$$
hence $uF(z)$ is an element of ${\rm Ann}_{E_R}I=E_S$. Therefore the restriction of $uF$ to
 ${\rm Ann}_{E_R}I$ gives a $p$-linear structure on $E_S$.
 
 It is an easy exercise, using Matlis duality, to check that $uF$ gives the trivial $p$-linear structure on $E_S$ if and only if $u\in I^{[p]}$.
On the other hand, Manuel Blickle (\cite[Chapter 3]{Blickle}) has shown that every $p$-linear structure on $E_S$ comes from an element in $(I^{[p]}\colon I)/I^{[p]}$. 
\end{remark}

\begin{remark}\label{completion}
Suppose now that $R$ is a regular local ring of positive characteristic, and
$S=R/I$, for some ideal $I$ in $R$. If we denote by $\widehat{R}$ and $\widehat{S}$ 
the completions of $R$ and $S$, respectively, then $E_S=E_{\widehat{S}}$,
and the $p$-linear structures of this module over $S$ and over $\widehat{S}$ can be identified. 
Using the fact that $\widehat{R}$ is flat over $R$, we deduce from the previous remark that
the $p$-linear structures on $E_S$ are in bijection with $(\widehat{I}^{[p]}\colon\widehat{I})/
\widehat{I}^{[p]}$, where $\widehat{I}=I\widehat{R}$. In particular, this set contains 
$(I^{[p]}\colon I)/I^{[p]}$. 
\end{remark}

\begin{example}
Let $R$ be an $n$-dimensional regular local ring of characteristic $p>0$, let $f\in R$ be nonzero and noninvertible, and $S=R/(f)$. We denote by $\frn$ and $\frm$ the maximal ideals in $S$ and $R$,
respectively. We have $E_R\simeq H^n_{\frm}(R)$ and 
$E_S\simeq{\rm Ann}_{E_R}(f)\simeq 
H^{n-1}_{\frn}(S)$ (the second isomorphism is a consequence of the exact sequence in the diagram below). 
Let $\Theta$ be the natural $p$-linear structure on $E_S$ induced by the Frobenius morphism on $S$. From the commutative diagram
$$
\xymatrix{
0  \ar[r] & H^{n-1}_{\frn}(S) \ar[r] \ar[d]^{\Theta} & H^n_{\frm}(R) \ar[r]^{f} \ar[d]^{f^{p-1}F} & H^n_{\frm}(R) \ar[r] \ar[d]^F & 0\\
0  \ar[r] & H^{n-1}_{\frn}(S) \ar[r]  & H^n_{\frm}(R) \ar[r]^{f}  & H^n_{\frm}(R) \ar[r] & 0,
}
$$
one sees immediately that $\Theta$ is the restriction of $f^{p-1}F$ to $E_S$.
Note also that for every $i\geq 1$, we have
$$\{u\in E_R\mid (f^{p-1}F)^i(u)=0\}\subseteq E_S.$$
Indeed, if $(f^{p-1}F)^i(u)=0$, then $f^{p^i-1}F^i(u)=0$, hence $f^{p^i}F^i(u)=F^i(fu)=0$,
and since $F$ is injective on $E_R$, we deduce $fu=0$, hence $u\in E_S$. 

We thus conclude that determining the HSL number of $(E_S,\Theta)$ is equivalent to determining the one of $(E_R,f^{p-1}F)$.
\end{example}

We recall that if $R$ is a regular $F$-finite domain of characteristic $p>0$, and
$f\in R$ is nonzero, then the test ideal $\tau(f^{m/p^e})$ has a simple
description  (see \cite[Lemma 2.1]{BMS1}):
\begin{equation}\label{eq_description_test}
\tau(f^{m/p^e})=(f^m)^{[1/p^e]}.
\end{equation}

\begin{proposition}\label{description_HSL1}
Let $R$ be a regular $F$-finite local ring of characteristic $p>0$, and let $f\in R$ be nonzero. If we put 
 $N_{i}=\{z\in E_R\mid (f^{p-1}F)^i(z)=0\}$ for $i\geq 1$, then
$${\rm Ann}_R(N_{i})=\tau(f^{(p^i-1)/p^i}).$$
\end{proposition}

\begin{proof}
Note first that since taking the test ideal commutes with completion, we may assume that
$R$ is complete. We then use \cite[Theorem~4.6]{Katzman} (and its proof) with
$I=(f)$ and $u=f^{p-1}$ to get that 
${\rm Ann}_RN_{i}$ is the smallest ideal $J_i$ of $R$ such that
\begin{enumerate}
\item[i)] $(f)\subseteq J_i$, and
\item[ii)]  $(f^{p-1})^{\frac{p^i-1}{p-1}}=f^{p^i-1}$ is contained in $J_i^{[p^i]}$.
\end{enumerate}
By definition, this says that ${\rm Ann}_RN_i=(f^{p^i-1})^{[1/p^i]}$, and the formula in the proposition follows from (\ref{eq_description_test}).
\end{proof}

\begin{corollary}
\label{connection-HSL-number-juming-number}
Let $R$ and $f$ be as in Proposition~\ref{description_HSL1}, with $f$ noninvertible.
The HSL number of $(E_R, f^{p-1}F)$ ${\rm (}$or equivalently, that of 
$(E_{R/(f)},\Theta)$${\rm )}$ is the smallest positive integer $\ell$ such that 
\[\tau( f^{1-\frac{1}{p^{\ell}}})=\tau(f^{1-\frac{1}{p^{\ell+i}}})\]
for all $i\geq 1$.
\end{corollary}

\begin{proof}
Since the HSL number of $(E_R, f^{p-1}F)$ does not change when we pass from $R$ to its completion, and taking test ideals commutes with completion, we may assume that $R$
is complete. In this case the assertion in the corollary follows from the definition of HSL
number for $(E_R,f^{p-1}F)$ via Matlis duality and
Proposition~\ref{description_HSL1}.
\end{proof}

Note that the condition in Corollary~\ref{connection-HSL-number-juming-number}
is equivalent to saying that there is no $F$-jumping number for $f$ in the interval
$\left(1-\frac{1}{p^{\ell}},1\right)$. We can now reformulate Theorem~\ref{thm_main1}
and Corollary~\ref{thm_main1_v2}, as follows.

\begin{theorem}\label{thm_main3}
Let $A$ be a finitely generated $\ZZ$-algebra, and $X$ a scheme of finite type over 
$\Spec A$, smooth of relative dimension $n$. If $\fra$ is a locally principal ideal on $X$, whose restriction to every fiber is everywhere nonzero, and if $Z$ is the closed subscheme defined by $\fra$, then there is a positive integer $N$
such that for every closed point $s\in \Spec A$, and every point $x$ in the fiber $Z_s$
of $Z$
over $s$, the HSL number of $(E_{\cO_{Z_s,x}},\Theta)$ is bounded above by $N$. 
\end{theorem}

\begin{proof}
The assertion follows from the above interpretation of HSL numbers, and 
Theorem~\ref{thm_main2} applied to $\lambda=1$, by noting two facts.
First, if $\frp$ is a prime ideal in a regular $F$-finite $R$ and if $\frb$ is an everywhere nonzero ideal in $R$, then the $F$-jumping numbers
of $\frb\cdot R_{\frp}$ are among the $F$-jumping numbers of $\frb$: this follows since taking multiplier ideals commutes with localization at $\frp$.
Second, we may take
$p_0=1$ (this follows by inspecting the proof of Theorem~\ref{thm_main1}).
\end{proof}

\begin{corollary}\label{cor_thm_main3}
Given $n$ and $M$, there is a positive integer $N=N(n,M)$ such that
for every $F$-finite field $k$ of characteristic $p>0$, and every nonzero polynomial
$f\in k[x_1,\ldots,x_n]$ with ${\rm deg}(f)\leq M$, the following holds:
for every prime ideal $\frp$ in $S=k[x_1,\ldots,x_n]/(f)$, the
HSL number of $(E_{S_{\frp}},\Theta)$ is bounded above by $N$. 
\end{corollary}

\begin{remark}
By inspecting the proof of Theorem~\ref{thm_main1}, we see that we may take 
$N(n,M)$ to be $1$ plus the dimension of the vector space of polynomials in 
$k[x_1,\ldots,x_n]$ of degree $\leq M$, that is $N(n,m)={{n+M}\choose {n}}+1$.
\end{remark}

\section{Examples}

In this section we give two examples, in order to to illustrate how the $F$-jumping numbers
vary when we reduce modulo various primes. In our computations of test ideals we will use the formula 
(\ref{formula_test_ideal}) in \S 2, describing ideals of the form $\frb^{[1/p^e]}$.
We begin by recalling the following
theorem due to Lucas (see \cite{Granville}).

\begin{theorem}\label{Lucas}
Let $p$ be a positive prime integer. Given two positive integers $m$ and $n$, if we write $m=m_ep^e+m_{e-1}p^{e-1}+\cdots m_0$ and $n=n_ep^e+n_{e-1}p^{e-1}+\cdots n_0$, with $0\leq m_i,n_i\leq p-1$, then
$${{m}\choose{n}}\equiv \prod_{i=0}^e{{m_i}\choose{n_i}}\ {\rm mod}\,p.$$
\end{theorem}

\begin{remark}\label{rem_Lucas}
\label{criterion-vanishing-binom}
It is clear from Lucas' theorem that $\binom{m}{n}$ is divisible by $p$ if and only if there is an $i$ such that $n_i>m_i$. 
\end{remark}

The following proposition shows that if we want to bound the HSL number of a polynomial, as 
in Corollary~\ref{cor_thm_main3}, we need indeed to bound the degree of the polynomial.

\begin{proposition}
\label{no-upper-bound-HSL}
Let $n$ be a positive integer, and set $a=2^{n+1}+1$ and $f=x^a+y^a+z^a$. For all prime numbers $p$ such that $p\equiv 2$ ${\rm (}$mod $a$${\rm )}$, the principal ideal $(f)$,
considered as an ideal in ${\mathbb F}_p\llbracket x,y,z\rrbracket$, has an $F$-jumping number in $\left(1-\frac{1}{p^n},1-\frac{1}{p^{n+1}}\right]$.
\end{proposition}

\begin{proof}
In order to prove the proposition, it is enough to show the following two assertions:
\begin{enumerate}
\item[(i)] $x^{a-3}$ lies in  $\tau(f^{1-\frac{1}{p^n}})$. 
\item[(ii)] $x^{a-3}$ does not lie in $\tau(f^{1-\frac{1}{p^{n+1}}})$.
\end{enumerate}

We first prove (i). It follows from (\ref{eq_description_test})
 that $\tau(f^{1-\frac{1}{p^n}})=(f^{p^n-1})^{[1/p^n]}$. Note that $a$ divides $p^n-2^n$,
 and we consider the following term in the expansion of $f^{p^n-1}$
\begin{align}
& \binom{p^n-1}{p^n-1-\frac{2}{a}(p^n-2^n)}\binom{\frac{2}{a}(p^n-2^n)}{\frac{1}{a}(p^n-2^n)}(x^a)^{p^n-1-\frac{2}{a}(p^n-2^n)}(y^a)^{\frac{1}{a}(p^n-2^n)}(z^a)^{\frac{1}{a}(p^n-2^n)}\notag\\
&=\binom{p^n-1}{p^n-1-\frac{2}{a}(p^n-2^n)}\binom{\frac{2}{a}(p^n-2^n)}{\frac{1}{a}(p^n-2^n)} x^{(a-3)p^n+p^n-1}y^{p^n-2^n}z^{p^n-2^n}\notag
\end{align}
It follows from the description for $(f^{p^e-1})^{[1/p^e]}$ given in (\ref{formula_test_ideal})
that if the two binomial coefficients in the above expression are not zero,
then (i) holds. Remark~\ref{criterion-vanishing-binom}
 implies that
${{p^n-1}\choose{i}}$ is not zero in ${\mathbf F}_p$ for every $i$ with $0\leq i\leq p^n-1$.
In particular, we have
$\binom{p^n-1}{p^n-1-\frac{2}{a}(p^n-2^n)}\neq 0$. For the other binomial coefficient, we write
\[\frac{2}{a}(p^n-2^n)=\sum^{n-1}_{j=0}\frac{2^{j+1}(p-2)}{a}p^{n-1-j};\ \frac{1}{a}(p^n-2^n)=\sum^{n-1}_{j=0}\frac{2^{j}(p-2)}{a}p^{n-1-j}.\]
Since $a=2^{n+1}-1$, we see that $\frac{2^{j+1}(p-2)}{a}<p$. Using  
again Remark \ref{criterion-vanishing-binom}, we deduce that 
$\binom{\frac{2}{a}(p^n-2^n)}{\frac{1}{a}(p^n-2^n)}\neq 0$, which completes the proof of (1).

We now prove (ii). It follows from the description of $\tau(f^{1-\frac{1}{p^{n+1}}})=
(f^{p^{n+1}-1})^{[1/p^{n+1}]}$
given by (\ref{formula_test_ideal}) that if $x^{a-3}$ lies in this ideal, then we have
a monomial $x^{ar}y^{as}z^{at}$ that appears with nonzero coefficient in the expansion
of $f^{p^{n+1}-1}$ such that 
$0\leq ar-(a-3)p^{n+1}\leq p^{n+1}-1$, $as\leq p^{n+1}-1$, and $at\leq p^{n+1}-1$.
In this case we have
$$as\leq p^{n+1}-1\ \Rightarrow\ s\leq \frac{p^{n+1}-1}{a}=\frac{p^{n+1}-2^{n+1}}{a}+\frac{2^{n+1}-1}{a}\ \Rightarrow\ s\leq \frac{p^{n+1}-2^{n+1}}{a},$$
where the last inequality holds since $a=2^{n+1}+1$ and $s$ is an integer. We similarly have 
$t\leq \frac{p^{n+1}-2^{n+1}}{a}$. Therefore
$$ar=a(p^{n+1}-1-s-t)\geq a(p^{n+1}-1)-a\cdot \frac{2}{a}(p^{n+1}-2^{n+1})=(a-2)p^{n+1}+(2^{n+2}-a)\geq (a-2)p^{n+1},$$
a contradiction.
We thus conclude that
$x^{a-3}\notin  \tau(f^{1-\frac{1}{p^{n+1}}})$, proving (ii).
\end{proof}

\begin{remark}
In \cite[Problem 1.05]{AIMworkshop}, Katzman asks the following question: given
$f\in\ZZ[x_1,\ldots,x_n]$, if $\alpha_p$ denotes the HSL number of
the injective hull of the residue field of ${\mathbf F}_p\llbracket x_1,\ldots,x_n\rrbracket/(f_p)$ with respect to the
natural $p$-linear structure $\Theta$, is $\limsup_{p\to\infty}\alpha_p=1$? 
Note that Proposition \ref{no-upper-bound-HSL} gives a negative
answer to this question. Indeed, since $2$ and $2^{n+1}+1$ are relatively prime, according to Dirichlet's theorem there are infinitely many prime numbers $p$ such that $p\equiv 2$ (mod $2^{n+1}+1$). Therefore Proposition~\ref{no-upper-bound-HSL} implies that given any $N$,
there is $f\in\ZZ[x_1,x_2,x_3]$ such that $\alpha_p>N$ for infinitely many primes $p$.

\end{remark}

\begin{example}\label{quintic}
Let $R=\mathbb{F}_p\llbracket x,y,z\rrbracket$ and $f=x^5+y^5+z^5$. In this case we have the following $F$-jumping numbers in $(0,1)$.
\begin{enumerate}
\item[(i)] When $p=2$ there are three $F$-jumping numbers in $(0,1)$: $\lambda_1=\frac{1}{4}$, $
\lambda_2=\frac{1}{2}$, and $ \lambda_3=\frac{3}{4}$.
\item[(ii)] When $p=3$ there are three $F$-jumping numbers in $(0,1)$: $\lambda_1=\frac{1}{3}$, $\lambda_2=\frac{2}{3}$, and $ \lambda_3=\frac{8}{9}$.
\item[(iii)] When $p=5$ there are four $F$-jumping numbers in $(0,1)$: $\lambda_1=\frac{1}{5}$, $\lambda_2=\frac{2}{5}$, $ \lambda_3=\frac{3}{5}$, and $\lambda_4=\frac{4}{5}$.
\item[(iv)] When $p\equiv 1$ (mod $5$), there are two $F$-jumping numbers in $(0,1)$: $\lambda_1=\frac{3}{5}$ and $\lambda_2=\frac{4}{5}$.
\item[(v)] When $p\equiv 2$ (mod $5$) and $p>2$ there are four $F$-jumping numbers in $(0,1)$: $\lambda_1=\frac{3}{5}-\frac{1}{5p}, \lambda_2=\frac{4}{5}-\frac{3}{5p}$, $\lambda_3=1-\frac{1}{p}$ and $ \lambda_4=1-\frac{1}{p^2}$.
\item[(vi)]  When $p\equiv 3$ (mod $5$) and $p>3$ there are four $F$-jumping numbers in $(0,1)$: $\lambda_1=\frac{3}{5}-\frac{4}{5p}, \lambda_2=\frac{4}{5}-\frac{2}{5p}$, $ \lambda_3=1-\frac{1}{p}$, and $\lambda_4=1-\frac{1}{p^2}$.
\item[(vii)]  When $p\equiv 4$ (mod $5$) there are three $F$-jumping numbers in $(0,1)$: 
$\lambda_1=\frac{3}{5}-\frac{7}{5p}$, $\lambda_2=\frac{4}{5}-\frac{6}{5p}$ and $ \lambda_3=1-\frac{1}{p}$.
\end{enumerate}
\end{example}

\begin{proof}
Since the calculations are a somewhat tedious, we only give the proof of (v).
The other cases are proved by entirely similar arguments.

It is straightforward to check that $\tau(f^{\frac{3}{5}-\frac{1}{5p}})=(x,y,z)$. We 
first show that $\tau(f^{\frac{3}{5}-\frac{1}{5p}-\frac{1}{p^t}})=R$ for all integers $t\geq 3$.
Note that
$\tau(f^{\frac{3}{5}-\frac{1}{5p}-\frac{1}{p^t}})=(f^{(3a+1)p^{t-1}-1})^{[1/p^t]}$.
We consider the following monomial that appears in the expansion of $f^{(3a+1)p^{t-1}-1}$:
\begin{align}
&(x^5)^{(3a+1)p^{t-1}-1-(2ap^{t-1}+4ap^{t-2}+4ap^{t-3})}(y^5)^{ap^{t-1}+2ap^{t-2}+2ap^{t-3}}(z^5)^{ap^{t-1}+2ap^{t-2}+2ap^{t-3}}\notag\\
&=x^{p^t-p^{t-3}(p^2-4p-8)-5}y^{p^t-2p^{t-2}-4p^{t-3}}z^{p^t-2p^{t-2}-4p^{t-3}}\notag
\end{align}
with coefficient $\binom{(3a+1)p^{t-1}-1}{2ap^{t-1}+4ap^{t-2}+4ap^{t-3}}\binom{2ap^{t-1}+4ap^{t-2}+4ap^{t-3}}{ap^{t-1}+2ap^{t-2}+2ap^{t-3}}$.
It follows from Remark \ref{criterion-vanishing-binom} that the two binomial coefficients are nonzero, and this implies that $1\in \tau(f^{\frac{3}{5}-\frac{1}{5p}-\frac{1}{p^t}})$. We have thus proved that the $F$-pure threshold of $f$ is  $\lambda_1=\frac{3}{5}-\frac{1}{5p}$.

Note that $\frac{4}{5}-\frac{3}{5p}=\frac{4a+1}{p}$ and it is easy to see that
 $\tau(f^{\frac{4}{5}-\frac{3}{5p}})=(f^{4a+1})^{[1/p]}=(x,y,z)^2$. We next show that $\tau(f^{\frac{4}{5}-\frac{3}{5p}-\frac{1}{p^t}})=(x,y,z)$ for all $t\geq 3$. We have  $\tau(f^{\frac{4}{5}-\frac{3}{5p}-\frac{1}{p^t}})=(f^{(4a+1)p^{t-1}-1})^{[1/p^t]}$. Consider the term in the expansion of 
 $f^{(4a+1)p^{t-1}-1}$
\[\binom{(4a+1)p^{t-1}-1}{2i}\binom{2i}{i}(x^5)^{(4a+1)p^{t-1}-1-2i}(y^5)^i(z^5)^i,\]
with $i=ap^{t-1}+2ap^{t-2}+2ap^{t-3}$. Note that $5i=p^t-2p^{t-2}-4p^{t-3}<p^t$ and 
\[p^t< 5((4a+1)p^{t-1}-1-2i)=2p^t-p^{t-3}(3p^2-4p-8)-1 <p^{2t}.\]
 Remark \ref{criterion-vanishing-binom} implies that $\binom{(4a+1)p^{t-1}-1}{2i}\binom{2i}{i}\neq 0$, and we conclude that $x\in \tau(f^{\frac{4}{5}-\frac{3}{5p}-\frac{1}{p^t}}) $. By symmetry, we also get  $y,z\in \tau(f^{\frac{4}{5}-\frac{3}{5p}-\frac{1}{p^t}})$. Therefore $\tau(f^{\frac{4}{5}-\frac{3}{5p}-\frac{1}{p^t}})=(x,y,z)$ for all $t\geq 1$. This shows that $\lambda_2=\frac{4}{5}-\frac{3}{5p}$ is indeed the second $F$-jumping number.

We now show that $\lambda_3=1-\frac{1}{p}$ is the third $F$-jumping number. We first prove that
\begin{equation}\label{eq_test_ideal20}
\tau(f^{1-\frac{1}{p}})=(f^{5a+1})^{[1/p]}=(x^2,y^2,z^2,xyz).
\end{equation}
To this end we consider $(x^5+y^5+z^5)^{5a+1}$. Since the exponent is less than $p$, none of the binomial coefficients that appear in the expansion of $(x^5+y^5+z^5)^{5a+1}$ is zero. The monomial $(x^5)^{3a+1}(y^5)^a(z^5)^a=x^{3p-1}y^{p-2}z^{p-2}$ shows that $x^2$
lies in $\tau(f^{1-\frac{1}{p}})$ (and by symmetry, so do $y^2$ and $z^2$). The monomial 
$(x^5)^{a+1}(y^5)^{2a}(z^5)^{2a}=x^{p+3}y^{2p-4}z^{2p-4}$ shows that $xyz$ lies in
$\tau(f^{1-\frac{1}{p}})$. We thus have $(x^2,y^2,z^2,xyz)\subseteq
\tau(f^{1-\frac{1}{p}})$.
In order to prove the reverse inclusion, it is enough to show that for every monomial
of the form 
$x^{5i}y^{5j}z^{5k}$, with $i+j+k=p-1$, we can not have $5i, 5j \leq 2p-1$ and $5k\leq p-1$.
Note that if these conditions hold, since $p\equiv 2$ (mod $5$), we have in fact
$5i, 5j\leq 2p-4$ and $5k\leq p-2$, hence $5p-5=5i+5j+5k\leq 5p-10$, a contradiction. This completes the proof of (\ref{eq_test_ideal20}).

To show that $1-\frac{1}{p}$ is the third $F$-jumping number, it suffices to prove that 
$xy$, $yz$, and $xz$ lie in  $\tau(f^{1-\frac{1}{p}-\frac{1}{p^t}})$ for all $t\geq 1$. 
Note that $\tau(f^{1-\frac{1}{p}-\frac{1}{p^t}})=(f^{p^t-p^{t-1}-1})^{[1/p^t]}$. Let us consider the
following term in the expansion of $f^{p^t-p^{t-1}-1}$:
\[\binom{p^t-p^{t-1}-1}{2i}\binom{2i}{i}(x^5)^{p^t-p^{t-1}-1-2i}(y^5)^i(z^5)^i,\]
with $i=2ap^{t-1}+2a\sum^{t-2}_{j=0}p^j$. It is easy to see that
\[p^t<5i<2p^t\,\, \text{and}\,\,0<5(p^t-p^{t-1}-1-2i)<p^t.\]
Remark \ref{criterion-vanishing-binom} implies that $\binom{p^t-p^{t-1}-1}{2i}\binom{2i}{i}\neq 0$, and we conclude that $yz$ lies in $\tau(f^{1-\frac{1}{p}-\frac{1}{p^t}})$ (and so do 
$xy$ and $xz$, by symmetry). This proves that $1-\frac{1}{p}$ is the third $F$-jumping number.

Our next goal is to show that $\lambda_4=1-\frac{1}{p^2}$ is the fourth $F$-jumping number. Let us first prove that 
\[\tau(f^{1-\frac{1}{p^2}})=(x,y,z)^3.\] 
Note that $\tau(f^{1-\frac{1}{p^2}})=(f^{p^2-1})^{[1/p^2]}$. By considering the nonzero term
in the expansion of $f^{p^2-1}$
$$\binom{p^2-1}{p^2-1-2i}
\binom{2i}{i}(x^5)^{p^2-1-2i}(y^5)^i(z^5)^i,\,\,\text{with}\,\,i=a(p+2),$$
we deduce $x^3$ lies in $\tau(f^{1-\frac{1}{p^2}})$ (and so do $y^3$ and $z^3$, by symmetry). 
Similarly, by considering the nonzero term
$$\binom{p^2-1}{p^2-1-(j+k)}\binom{j+k}{j}(x^5)^{p^2-1-(j+k)}(y^5)^j(z^5)^k,\,\,\text{with}\,\,j=2a(p+1),\,k=a(p+1)$$ we obtain that $x^2y$ lies $\tau(f^{1-\frac{1}{p^2}})$ (and 
by symmetry, so do $xy^2$, $y^2z$, $yz^2$, $x^2z$, and $xz^2$). Finally, the nonzero term 
$$\binom{p^2-1}{2i}\binom{2i}{i}(x^5)^{p^2-1-2i}(y^5)^i(z^5)^i,\,\,\text{with}\,\,i=2a(p+1)$$ implies that $xyz$ is contained in $\tau(f^{1-\frac{1}{p^2}})$, and therefore
$(x,y,z)^3\subseteq\tau(f^{1-\frac{1}{p^2}})$.

In order to check the reverse inclusion, it is enough to show that for every monomial
of the form 
$x^{5i}y^{5j}z^{5k}$, with $i+j+k=p^2-1$, the following hold:
\begin{enumerate}
\item[(a)] If two of $5i$, $5j$, and $5k$ are $\leq p^2-1$, then the third one is $\geq 3p^2$.
\item[(b)] If two of $5i$, $5j$, and $5k$ are $\leq 2p^2-1$, then the third one is $\geq p^2$.
\end{enumerate}
For (a), note that if $5i,5j\leq p^2-1$, then since $p\equiv 2$ (mod $5$), we have in fact $5i, 5j\leq p^2-4$, hence $5k\geq 5(p^2-1)-2(p^2-4)=3p^2+3$. We argue similarly for (b): if 
$5i, 5j\leq 2p^2-1$, then in fact $5i, 5j\leq 2p^2-3$, and therefore
$5k\geq 5(p^2-1)-2(2p^2-3)=p^2+1$. We have thus shown that
$\tau(f^{1-\frac{1}{p^2}})=(x,y,z)^3$. 

In order to complete the proof of the fact that $\lambda_4$ is the fourth $F$-jumping number, 
it is enough to show that $x^2, y^2,z^2 \in \tau(f^{1-\frac{1}{p^2}-\frac{1}{p^t}})$ for all $t\geq 3$. Note that $\tau(f^{1-\frac{1}{p^2}-\frac{1}{p^t}})=(f^{p^t-p^{t-2}-1})^{[1/p^t]}$. It is clear that
\[p^t-p^{t-2}-1=(5a+1)p^{t-1}+5ap^{t-2}+(5a+1)\sum^{t-3}_{j=0}p^j.\]
We consider the term
\[\binom{p^t-p^{t-2}-1}{2i}\binom{2i}{i}(x^5)^{p^t-p^{t-2}-1-2i}(y^5)^i(z^5)^i,\]
where $i=ap^{t-1}+2ap^{t-2}+2a\sum^{t-3}_{j=0}p^j$.
It is straightforward to check that $5i<p^t$ and 
$2p^t < 5(p^t-p^{t-2}-1-2i) < 3p^t$.
Since it follows from Remark \ref{criterion-vanishing-binom} that $\binom{p^t-p^{t-2}-1}{2i}\binom{2i}{i} \neq 0$, we conclude that
 $x^2$ lies in  $\tau(f^{1-\frac{1}{p^2}-\frac{1}{p^t}})$ (and so do $y^2$ and $z^2$, by symmetry). This shows that $1-\frac{1}{p^2}$ is the fourth $F$-jumping number.   

In order to show that there is no $F$-jumping number in $(1-\frac{1}{p^2},1)$, we need to prove that $(x,y,z)^3\subseteq \tau(f^{1-\frac{1}{p^t}})$ for all $t\gg 0$. It is clear that 
$$p^t-1=(p-1)\sum^{t-1}_{j=0}p^j=(5a+1)\sum^{t-1}_{j=0}p^j.$$
 The following nonzero term in the expansion of $f^{p^t-1}$:
 $$\binom{p^t-1}{2i}\binom{2i}{i}(x^5)^{p^t-1-2i}(y^5)^i(z^5)^i,\,\,\text{with}\,\,i=a\sum^{t-1}_{j=0}p^j$$ shows that $x^3$ lies in $\tau(f^{1-\frac{1}{p^t}})$ (and so do $y^3$ and $z^3$, by symmetry). The nonzero term
 $$\binom{p^t-1}{j+k}\binom{j+k}{j}(x^5)^{p^t-1-(j+k)}(y^5)^j(z^5)^k,\,\,\text{with}\,\,j=2a\sum^{t-1}_{n=0}p^n,\,k=a\sum^{t-1}_{n=0}p^n$$ shows that $x^2y$ lies in $\tau(f^{1-\frac{1}{p^t}})$ (so do $xy^2$, $x^2z$, $xz^2$, $y^2z$, and $yz^2$, by symmetry). Finally, the 
 nonzero term 
 $$\binom{p^t-1}{2i}\binom{2i}{i}(x^5)^{p^t-1-2i}(y^5)^i(z^5)^i,\,\,\text{with}\,\,i=2a\sum^{t-1}_{n=0}p^n$$ shows that $xyz$ lies in $\tau(f^{1-\frac{1}{p^t}})$. This proves that there 
 are no  $F$-jumping numbers in the interval $(1-\frac{1}{p^2},1)$.
\end{proof}

\begin{remark}
Let $R$ and $f$ be as in Example~\ref{quintic}.
The description of the $F$-jumping numbers in this example allows us 
to compute via Corollary~\ref{connection-HSL-number-juming-number}
the HSL number of $(E_{R/(f)},\Theta)$. We see that this is equal to $2$ if
$p\equiv 2$ or $3$ (mod $5$), and it is equal to $1$, otherwise.
\end{remark}

\bigskip
\bigskip

\providecommand{\bysame}{\leavevmode \hbox \o3em
{\hrulefill}\thinspace}

\end{document}